\documentclass[reqno]{amsart}

\usepackage[english]{babel}
\usepackage{amssymb,amsmath,amscd,amsfonts,amsthm,bbm,color}
\usepackage{hyperref}
\usepackage{graphicx} 
\usepackage{tikz,pgfplots}

\newcommand{\tdl}{\tilde\delta}

\newcommand{\xII}{\bs x_{{\cal I}, \widetilde{\cal I}}} 
\newcommand{\zII}{z_{{\cal I}, \widetilde{\cal I}}}

\DeclareMathOperator{\rank}{rank}
\DeclareMathOperator{\dom}{dom}

\newcommand{\1}{\mathbbm 1}

\newcommand{\RR}{\mathbb R}

\newcommand{\CC}{\mathbb{C}}

\newcommand{\EE}{\mathbb E}

\newcommand{\cal}{\mathcal}
    
\theoremstyle{plain} 

\newtheorem{theorem}{Theorem}[section]
\newtheorem{lemma}{Lemma}[section]

\newtheorem{proposition}{Proposition}[section]

\DeclareMathOperator{\tr}{Tr}

\DeclareMathOperator*{\support}{supp}
\newcommand{\bs}{\boldsymbol}
\newcommand{\toaslong}{\xrightarrow[n\to\infty]{\text{a.s.}}}

\title[Limiting spectral measure of large random matrices] 
{Analysis of the limiting spectral measure \\ 
of large random matrices \\ 
of the separable covariance type} 

\author[Couillet and Hachem]{Romain Couillet and Walid Hachem} 

\thanks{(R.~Couillet) Sup\'elec, Plateau de Moulon, 91192 
Gif-sur-Yvette, France. \\ E-mail: \texttt{romain.couillet@supelec.fr}. \\  
(W.~Hachem) CNRS LTCI; Telecom ParisTech, 46 rue Barrault, 75013, Paris, 
France. \\ E-mail: \texttt{walid.hachem@telecom-paristech.fr}. \\
This work is partially funded by the French Agence Nationale de la
Recherche under the program ``Mod\`eles Num\'eriques'' under the grant
ANR-12-MONU-0003 (project DIONISOS)}

\begin{document}

\begin{abstract}
Consider the random matrix $\Sigma = D^{1/2} X \widetilde D^{1/2}$ where 
$D$ and $\widetilde D$ are deterministic Hermitian nonnegative matrices with 
respective dimensions $N \times N$ and $n \times n$, and where $X$ is a 
random matrix with independent and identically distributed centered elements
with variance $1/n$. Assume that the dimensions $N$ and $n$ grow to infinity
at the same pace, and that the spectral measures of $D$ and $\widetilde D$ 
converge as $N,n \to\infty$ towards two probability measures. Then it
is known that the spectral measure of $\Sigma\Sigma^*$ converges towards
a probability measure $\mu$ characterized by its Stieltjes Transform. \\
In this paper, it is shown that $\mu$ has a density away from zero, this 
density is analytical wherever it is positive, and it behaves in most cases as 
$\sqrt{|x - a|}$ near an edge $a$ of its support. In addition, a complete characterization of
the support of $\mu$ is provided. \\ 
Aside from its mathematical interest, the analysis underlying these results finds important applications in a 
certain class of statistical estimation problems. 
\end{abstract}



\date{6 October 2014}  

\maketitle

\section{Introduction and problem statement}
\label{sec:intro}
Consider the $N \times n$ random matrix 
$\Sigma_n = D_n^{1/2} X_n \widetilde D_n^{1/2}$ 
where $X_n$ is a $N \times n$ real or complex random matrix having independent 
and identically distributed elements with mean zero and variance $1/n$, the
$N\times N$ matrix $D_n$ is determinisitic, Hermitian and nonnegative, and 
the $n \times n$ matrix $\widetilde D_n$ is also deterministic, Hermitian 
and nonnegative. We assume that $n\to\infty$ and $N/n \to c > 0$,
and we denote this asymptotic regime as ``$n\to\infty$''. 
We also assume that the spectral measures of $D_n$ and $\widetilde D_n$
converge respectively towards the probability measures $\nu$ and $\tilde\nu$
as $n\to\infty$. We assume that $\nu \neq \bs d_0$ and 
$\tilde\nu \neq \bs d_0$ where $\bs d_x$ the Dirac measure at $\{x\}$. 
Many contributions showed that the spectral measure of $\Sigma_n \Sigma_n^*$ 
converges to a deterministic probability measure $\mu$ and provided a 
characterization of this limit measure under various assumptions 
\cite{gir-rand-determ90, shl-amalgam-96, BKV96, HLN06}, 
the weakest being found in \cite{zhang06}.   
In this work, we show that $\mu$ has a density away from zero, this density is 
analytical wherever it is positive, and it behaves as $\sqrt{|x - a|}$ near
an edge $a$ of its support for a large class of measures $\nu$, $\tilde{\nu}$.
We also provide a complete characterization of this support along with a thorough analysis of the master equations relating $\mu$ to $\nu$ and $\tilde{\nu}$. 
To that end, we follow the general ideas already provided in the classical 
paper of Marchenko and Pastur \cite{MarPas67} and further developed in 
\cite{sil-choi95} and \cite{DozSil07b}. \\ 
In \cite{sil-choi95}, Silverstein and Choi performed this study in the so called
sample covariance matrix 
case where $\widetilde D_n = I_n$. 
The outline of the present article closely follows that of \cite{sil-choi95} although at multiple occasions our proofs depart from those of \cite{sil-choi95}, making the article more self-contained. In particular, while Silverstein and Choi benefited from the existence of an explicit expression for the inverse of the Stieltjes transform of $\mu$ when $\widetilde D_n = I_n$, this is no longer the case in the general setting requiring the use of more fundamental analytical tools.
In the setting of \cite{sil-choi95}, it has been further shown in 
\cite{BaiSil98} that under some conditions, no closed interval outside the 
support of $\mu$ contains an eigenvalue of $\Sigma_n\Sigma_n^*$, with 
probability one, for all large $n$. In \cite{bai-sil-separation}, a finer
result on the so called ``exact separation'' of the eigenvalues of 
$\Sigma_n\Sigma_n^*$ between the connected components of the support of 
$\mu$ is shown. 
Recently, it has been discovered that the characterization in \cite{sil-choi95} of the support of $\mu$ and the results on the master equations relating $\mu$ to $\nu$,
beside their own interest, lead in conjunction with the results of 
\cite{BaiSil98, bai-sil-separation} to the design of consistent statistical 
estimators of some linear functionals of the eigenvalues of $D_n$ or 
projectors on the eigenspaces of this matrix. Such estimators have been 
developed by Mestre in \cite{mes-it08,mes-sp08}, the initial idea dating 
back to the work of Girko (see \emph{e.g.} \cite{Gir-canonical}). \\
In \cite{DozSil07b}, Brent Dozier and Silverstein studied the properties of the
limit spectral measure of the so called ``Information plus Noise'' ensemble. 
A first result on the absence of eigenvalues outside the support of the limit 
spectral measure has been established in \cite{bai-sil-rmta12}. In 
\cite{lou-val-ejp11, hlmnv-rmta12,val-lou-mes-it12} other separation results
as well as statistical estimation algorithms along the lines of 
\cite{mes-it08,mes-sp08} were proposed. \\
Turning to the separable covariance matrix ensemble of interest here, the 
absence of eigenvalues
outside the support of $\mu$ has been established by Paul and Silverstein 
in \cite{paul-sil09} without characterizing this support. 
The results of this paper therefore complement those of \cite{paul-sil09}. 
More importantly, similar to the case $\widetilde D_n=I_n$, these results are a necessary first step to 
devise statistical estimation algorithms of \emph{e.g.}
linear functionals of the eigenvalues of one of the matrices $D_n$ or 
$\widetilde D_n$. Work on this subject is currently in progress. \\
Finally, it has been noticed in the large random matrix community that there
is an intimate connection between the square root behavior of the density 
of the limit spectral measure at the edges of the support and the 
Tracy-Widom fluctuations of the eigenvalues near those edges 
(see \cite{kar-ap7} dealing with the sample covariance matrix case). 
It can be conjectured that such behaviour still holds (with some assumptions
on the probability law of the elements of $X_n$)
in the separable covariance case
considered here. In this regard, Theorem \ref{sq-root} may help guessing the   
exact form of the Tracy-Widom law at the edges of the support of $\mu$.

We now recall the results describing the asymptotic behavior of the spectral
measure of $\Sigma_n\Sigma_n^*$. 

\subsection{The master equations} 
We recall that the Stieltjes Transform of a probability measure $\pi$ on $\RR$ 
is the function
\[
f(z) = \int \frac{1}{t-z} \pi(dt) 
\]
defined on $\CC_+$. The function $f(z)$ is \emph{i)} holomorphic on 
$\CC_+ = \{ z\, : \, \Im(z) > 0 \}$, \emph{ii)} it satisfies $f(z) \in \CC_+$ 
for any $z \in \CC_+$, and \emph{iii)} 
$\lim_{y\to\infty} | y f(\imath y) | = 1$. In addition, if 
$\pi$ is supported by $\RR_+ = [0,\infty)$, then \emph{iv)} $z f(z) \in \CC_+$ 
for any $z \in \CC_+$. 
Conversely, it is well known that any function $f(z)$ satisfying 
\emph{i)}--\emph{iv)} is the Stieltjes Transform of a probability measure 
supported by $\RR_+$ \cite{Krein77}. Finally, observe that the Stieltjes 
Transform of $\pi$ can be trivially extended from $\CC_+$ to 
$\CC - \support(\pi)$ where $\support(\pi)$ is the support of $\pi$. \\
In this paper, a small generalization of this result will be needed
\cite[Appendix A]{Krein77}: The three following statements are equivalent: 
\begin{itemize}
\item The function $f(z)$ satisfies the properties \emph{i)}, \emph{ii)}, 
 and \emph{iv)}, 
\item It admits the representation 
\[
f(z) = a + \int_0^\infty \frac{1}{t-z} \pi(dt) 
\] 
where $a \geq 0$ and where $\pi$ is a Radon positive measure on $\RR_+$ such
that $0 < \int_0^\infty (1+t)^{-1} \pi(dt) < \infty$, 
\item The function $f(z)$ satisfies the properties \emph{i)} and \emph{ii)},
and furthermore, it is analytical and nonnegative on the negative real 
axis $(-\infty, 0)$. 
\end{itemize} 
 
We now recall the first order result. 
\begin{proposition}[\cite{zhang06}, see also \cite{hachem-loubaton-najim07} for
similar notations] 
\label{1st-order} 
Let the probability measures $\nu\neq \bs d_0$ and $\tilde\nu\neq \bs d_0$ be 
the limit spectral measures of the matrices $D_n$ and $\widetilde D_n$ 
respectively. 
For any $z \in \CC_+$, the system of equations 
\begin{align}
\delta &= c \int \frac{t}{-z(1+\tdl t)} \nu(dt) 
\label{delta} \\
\tdl &= \int \frac{t}{-z(1+\delta t)} \tilde\nu(dt) 
\label{deltat} 
\end{align} 
admits a unique solution $(\delta, \tdl) \in \CC_+^2$. 
Let $\delta(z)$ and $\tdl(z)$ be these solutions. The function 
\begin{equation}
\label{ts} 
m(z) = \int \frac{1}{-z(1+\tdl(z)t)} \nu(dt) , 
\quad z \in \CC_+  
\end{equation} 
is the Stieltjes Transform of a probability measure $\mu$ supported by $\RR_+$.
The function 
\[
\tilde m(z) = \int \frac{1}{-z(1+\delta(z)u)} \tilde\nu(du),  
\quad z \in \CC_+  
\]
is the Stieltjes Transform of the probability measure 
$\tilde\mu = c\mu + (1-c) \bs d_0$. Moreover, denoting by $\mu_n$ the  
spectral measure of $\Sigma_n \Sigma_n^*$ and by 
$\tilde\mu_n = (N/n) \mu_n + (1-N/n) \bs d_0$ the spectral measure of 
$\Sigma_n^* \Sigma_n$, it holds that 
\[
\int \varphi(\lambda) \mu_n(d\lambda) 
\toaslong 
\int \varphi(\lambda) \mu(d\lambda) \quad \text{and} \quad 
\int \varphi(\lambda) \tilde\mu_n(d\lambda) 
\toaslong 
\int \varphi(\lambda) \tilde\mu(d\lambda) 
\]
for any continuous and bounded real function $\varphi$.  
\end{proposition} 

Before going further, we collect some simple facts and identities that will be 
often used in the paper: 

\begin{itemize} 
\item 
Define the function 
\begin{equation}
\label{tdelta-seul} 
F(\tdl, z) = \int\frac{t}
{\displaystyle{-z + 
ct \int \frac{u}{1+u\tilde\delta} \nu(du)} } 
\tilde\nu(dt) \ - \tdl , \quad (\tdl, z) \in \CC_+^2  . 
\end{equation} 
By plugging Equation \eqref{delta} into Equation \eqref{deltat}, we obtain
that the function $\tilde\delta(z)$ can also be defined as the unique solution 
of the equation $F(\tdl, z) = 0$. It will be sometimes more convenient to work 
on this equation instead of the ``split'' form \eqref{delta}--\eqref{deltat}. 

\item The functions $m(z)$ and $\tilde m(z)$ satisfy the identities 
\begin{equation} 
\label{m-delta} 
\begin{split} 
m(z) &= \int \frac{1+\tdl(z)t - \tdl(z)t}{-z(1+\tdl(z)t)} \nu(dt)
= -z^{-1} - c^{-1} \delta(z)\tdl(z) , \quad \text{and} \\ 
\tilde m(z) &= -z^{-1} - \delta(z)\tdl(z) . 
\end{split} 
\end{equation} 

\item For any $z_1, z_2 \in \CC_+$, define 
\begin{align}
\begin{split} 
\gamma(z_1, z_2) &= 
c \int \frac{t^2}
{z_1 z_2(1+\tdl(z_1)t)(1+\tdl(z_2)t)} \nu(dt), \quad \text{and} \\ 
\tilde\gamma(z_1, z_2) &= 
\int \frac{t^2}
{z_1 z_2(1+\delta(z_1)t)(1+\delta(z_2)t)} \tilde\nu(dt) 
\end{split}
\label{gamma} 
\end{align}
(since $|(1+\tdl(z_1)t)(1+\tdl(z_2)t)| \geq \Im\tdl(z_1) \Im\tdl(z_2) t^2$ and 
$|(1+\delta(z_1)t)(1+\delta(z_2)t)| \geq \Im\delta(z_1) \Im\delta(z_2) t^2$, 
the integrability is guaranteed). By the definition of $\tdl(z)$, we have
\begin{align*}
\tdl(z_1) - \tdl(z_2) &= 
\int \frac{ (z_1 - z_2) t + (z_1\delta(z_1) - z_2 \delta(z_2))t^2}
{z_1 z_2( 1 + \delta(z_1) t)(1 + \delta(z_2) t)} \tilde\nu(dt)   
\end{align*}
and by developing the expression of $z_1\delta(z_1) - z_2\delta(z_2)$ 
using~\eqref{delta}, we obtain
\begin{multline} 
\label{diff} 
(1 - z_1 z_2 \gamma(z_1, z_2) \tilde\gamma(z_1, z_2) ) 
(\tdl(z_1) - \tdl(z_2)) \\ = 
(z_1 - z_2) 
\int \frac{t}
{z_1 z_2( 1 + \delta(z_1) t)(1 + \delta(z_2) t)} \, \tilde\nu(dt)  . 
\end{multline} 

Similarly, 
\begin{align*}
\gamma(z, z^*) &= 
c \int \frac{t^2}{|z|^2 |1+\tdl(z)t|^2} \nu(dt), \quad \text{and} \\ 
\tilde\gamma(z, z^*) &= 
\int \frac{t^2}{|z|^2 |1+\delta(z)t|^2} \tilde\nu(dt) 
\end{align*} 
are defined for any $z\in\CC_+$ since $|z (1+\tdl(z)t)|^2 \geq 
(\Im (z\tdl(z)))^2 t^2$. By a derivation similar to above, we have 
for any $z\in\CC_+$ 
\begin{align*}
\Im\tdl(z) &= \frac{\tdl(z) - \tdl(z)^*}{2\imath} \\
&= \Im(z\delta(z)) \tilde\gamma(z, z^*) + 
\Im z \int \frac{t}{|z|^2 |1+\delta(z) t|^2} \tilde\nu(dt) 
\end{align*} 
By writing $\Im(z\delta(z)) = (z\delta(z) - z^*\delta(z)^*)/(2\imath)$ and
by developing this expression using~\eqref{delta}, we get 
\begin{equation}
\label{gamma(z,z*)} 
(1 - | z|^2  \gamma(z, z^*) \tilde\gamma(z, z^*) ) \Im \tdl(z) = 
\Im z 
\int \frac{t}{|z|^2 | 1 + \delta(z) t)|^2} \, \tilde\nu(dt)  . 
\end{equation} 
On $\CC_+$, $\Im\tdl(z) > 0$. Moreover, the integral at the right hand side is 
strictly positive. Hence 
\[
\forall\, z \in \CC_+, 
\quad 1 - |z|^2 \gamma(z, z^*) \tilde\gamma(z, z^*) > 0 . 
\]
This inequality will be of central importance in the sequel. 
\end{itemize}

The two measures introduced by the following proposition share many 
properties with $\mu$ as it will be seen below. They will play an essential 
role in the paper. 
\begin{proposition} 
\label{class-S} 
The functions $\delta(z)$ and $\tdl(z)$ admit the representations
\[
\delta(z) = \int_0^\infty \frac{1}{t-z} \rho(dt) 
\quad \text{and} \quad 
\tilde\delta(z) = \int_0^\infty \frac{1}{t-z} \tilde\rho(dt),  
\quad z \in \CC_+ 
\]
where $\rho$ and $\tilde\rho$ are two Radon positive measures on $\RR_+$ 
such that 
\[
0 < \int_0^\infty \frac{1}{1+t} \rho(dt) < \infty 
\quad \text{and} \quad 
0 < \int_0^\infty \frac{1}{1+t} \tilde\rho(dt) < \infty  .
\]
\end{proposition} 
\begin{proof} 
One can observe that the function $F(\tdl, z)$ defined in \eqref{tdelta-seul} 
is holomorphic on $\CC_+^2$. Fixing $z_0 \in \CC_+$, a small calculation shows 
that  
\begin{multline*} 
\Bigl| \frac{\partial F}{\partial\tdl} (\tdl, z_0) \Bigr| = 
| 1 - z_0^2 \gamma(z_0,z_0) \tilde\gamma(z_0,z_0) | \\ 
\geq 1 - | z_0^2 \gamma(z_0,z_0) \tilde\gamma(z_0,z_0) | 
\geq 1 - | z_0| ^2 \gamma(z_0,z_0^*) \tilde\gamma(z_0,z_0^*) > 0 
\end{multline*} 
by Inequality \eqref{gamma(z,z*)}.  
The holomorphic implicit function theorem 
\cite[Ch.~1, Th.~7.6]{frit-grau-(livre)02} shows then that $\tilde\delta(z)$ 
is holomorphic in a neighborhood of $z_0$. Since $z_0$ is chosen arbitrarily
in $\CC_+$, we get that $\tilde\delta(z)$ is holomorphic in $\CC_+$. 
Recall that $\Im\tdl(z) > 0$ on $\CC_+$. Since we furthermore have 
\[
\Im(z \tdl(z)) = \Im\delta(z) \int \frac{t^2}{| 1 + \delta(z) t|^2} 
       \tilde\nu(dt) > 0
\]
on $\CC_+$, we get the representation 
\[
\tilde\delta(z) = \tilde a + \int \frac{1}{t-z} \tilde\rho(dt) 
\]
where $\tilde a \geq 0$ and where $\tilde\rho$ satisfies the properties 
given in the statement. Let us show that $\tilde a = 0$. Observe that 
$\tdl(x) \downarrow \tilde a$ when $x$ is a real negative number converging to 
$-\infty$. By a continuation argument, $F(\tdl(x), x) = 0$ for any 
negative value of $x$. As $x\to -\infty$, we get by the monotone convergence
theorem 
\[
I(\tdl(x)) = \int \frac{u}{1 + u\tdl(x)} \nu(du) \uparrow 
I(\tilde a) = \int \frac{u}{1 + u \tilde a} \nu(du) \in (0, \infty] . 
\]
When $x < 0$ is far enough from zero, $I(\tdl(x)) \geq C$ where $C > 0$ is 
a constant, and the Dominated Convergence Theorem (DCT) shows that 
\[
\tdl(x) = \int \frac{t}{-x + ct I(\tdl(x))} \nu(dt) 
\xrightarrow[x\to-\infty]{} 0 . 
\]
A similar argument can be applied to $\delta(z)$. 
\end{proof}

\section{Some elementary properties of $\mu$} 
\label{elementary} 

Before entering the core of the paper, it might be useful to establish
some elementary properties of $\mu$. 

In the asymptotic regime where $N$ is fixed and $n \to\infty$, the matrix
$\Sigma_n\Sigma_n^* - (n^{-1} \tr\widetilde D_n) D_n$ will converge to zero 
when the assumptions of the law of large numbers are satisfied. In our 
asymptotic regime, the following result can therefore be expected. 
Note that this result has its own interest and has no relation with the rest 
of the paper. 
\begin{proposition}
\label{extreme-c} 
Assume that $M_\nu = \int t \nu(dt)$ and $M_{\tilde\nu} = \int t\tilde\nu(dt)$
are both finite. Then 
\[
\mu(dt) \Rightarrow \nu( M_{\tilde\nu}^{-1} \, dt) \quad 
\text{as} \quad c \to 0  
\]
where $\Rightarrow$ denotes the weak convergence of probability measures. 
\end{proposition}
\begin{proof} 
For any $u \geq 0$ and any $z \in \CC_+$, 
$| z(1+\tdl(z)u) | \geq  \Im(z(1+\tdl(z)u))  
\geq \Im(z)$, hence $|\delta(z)| \leq c M_\nu / \Im(z)$, which implies that
$\delta(z) \to 0$ as $c\to 0$. Similarly, $| z(1+\delta(z)t) | \geq  \Im(z)$
for any $t \geq 0$ and any $z \in \CC_+$, hence 
$\tdl(z) \to - M_{\tilde\nu} / z$ by the DCT. Invoking the DCT again, we get 
that 
\[
m(z) \xrightarrow[c\to 0]{} 
\int \frac{1}{M_{\tilde\nu} t - z} \nu(dt) 
= 
\int \frac{1}{t - z} \nu(M_{\tilde\nu}^{-1} \, dt) 
\]
which shows the result. 
\end{proof} 

We now characterize $\mu(\{0\})$. Intuitively, 
$\rank(\Sigma_n) \simeq \min[ N(1-\nu(\{0\})), n(1-\tilde\nu(\{0\})) ]$ 
and $\mu(\{0\}) \simeq 1 - \rank(\Sigma_n)/N$ for large $n$.  
The following result is therefore expected: 
\begin{proposition}
\label{mu(0)}
$\mu(\{ 0 \}) = 
1 - \min[ 1-\nu(\{0\}), c^{-1}(1-\tilde\nu(\{0\})) ]$. 
\end{proposition} 
\begin{proof}
From the general expression of a Stieltjes Transform of a probability measure, 
it is easily seen using the DCT that 
$\mu(\{0\}) = \lim_{y\downarrow 0} (-\imath y m(\imath y))$. 
Moreover, since 
$| y (t - \imath y)^{-1}| \leq (t^2 + 1)^{-1/2}$ when $|y| \leq 1$, 
the DCT and Proposition \ref{class-S} show that 
$\tilde\rho(\{0\}) = \lim_{y\downarrow 0} (-\imath y \tdl(\imath y))$. \\ 
Let us write $\nu = \nu(\{0\}) \bs d_0 + \nu'$ and 
$\tilde\nu = \tilde\nu(\{0\}) \bs d_0 + \tilde\nu'$, and let us assume that
$1-\nu(\{0\}) < c^{-1}(1-\tilde\nu(\{0\}))$, or equivalently, that 
$\nu'(\RR_+) < c^{-1} \tilde\nu'(\RR_+)$. In this case, we will show that 
$\tilde\rho(\{0\}) > 0$. That being true, we get 
\[
\mu(\{ 0 \}) = 
\lim_{y\downarrow 0} (-\imath y m(\imath y) ) = \nu(\{0\}) + 
\lim_{y\downarrow 0} 
        \int \frac{1}{1 + \tdl(\imath y) t} \nu'(dt) 
= \nu(\{0\}) 
\]
(since $\Re(\tdl(\imath y)) > 0$, see below, the integrand above 
is bounded in absolute value by $1$, and furthermore, it converges
to $0$ for any $t > 0$ due to the fact that $\tilde\rho(\{0\}) > 0$). \\
We assume that $\tilde\rho(\{0\}) = 0$ and raise a contradiction. 
The equation $F(\tdl, \imath y ) = 0$ for $y > 0$ can be rewritten as  
\[
\int\frac{t}
{\displaystyle{-\imath y\tdl(\imath y) + 
ct \int \frac{u\tdl(\imath y)}{1+u\tdl(\imath y)} \nu'(du)} } 
\tilde\nu'(dt) 
= 1 . 
\]
We have 
\[
\Re( \tdl(\imath y) ) = \Re \int \frac{1}{t-\imath y} \tilde\rho(dt)
= \int \frac{t}{t^2 + y^2} \tilde\rho(dt) > 0,  
\]
and $\lim_{y\to 0} \Re(\tdl(\imath y) ) \in (0, \infty]$ by the 
monotone convergence theorem. Let 
\[
I(y) = \int \frac{u\tdl(\imath y)}{1+u\tdl(\imath y)} \nu'(du) . 
\]
Writing $\tdl = \tdl(\imath y)$, we have 
\[
\Re(I(y)) = \int 
\frac{u (\Re\tdl) (1+u \Re\tdl) + (u \Im\tdl)^2}
{(1 + u \Re\tdl)^2 + (u \Im\tdl)^2} \nu'(du) 
\]
whose $\liminf$ is positive as $y \downarrow 0$. Furthermore, we have for 
$y > 0$  
\[
\Re(-\imath y\tdl(\imath y)) = 
\Re\int \frac{-\imath y}{t-\imath y} \tilde\rho(dt) 
= \int \frac{y^2}{t^2+y^2} \tilde\rho(dt) > 0  
\]
hence $\liminf_{y\downarrow 0} | -\imath y \tdl(\imath y) + c t I(y) | 
\geq c t \liminf_{y\downarrow 0} \Re I(y)$. Consequently, we have by 
the assumption $\tilde\rho(\{0\}) = 0$ and the DCT
\[
\int\frac{t}
{-\imath y\tdl(\imath y) + ct I(y)} 
\tilde\nu'(dt) 
- 
\frac{\tilde\nu'(\RR_+)}{c I(y)} 
\xrightarrow[y\downarrow 0]{} 0 . 
\]
This shows that $\lim_{y\downarrow 0} I(y) = c^{-1} \tilde\nu'(\RR_+)$. 
But since $\Re(\tdl(\imath y)) > 0$, 
$| u\tdl(\imath y) ( 1 + u\tdl(\imath y))^{-1} | \leq 1$ for $u\geq 0$ hence
$| I(y) | \leq \nu'(\RR_+)$. Therefore, 
$c^{-1} \tilde\nu'(\RR_+) \leq \nu'(\RR_+)$ which contradicts the assumption. 
\\
If $\nu'(\RR_+) > c^{-1} \tilde\nu'(\RR_+)$, we replace $\mu$, $m(z)$ and
$\tdl(z)$ with $\tilde\mu$, $\tilde m(z)$ and $\delta(z)$ respectively in the
previous argument. \\
To deal (briefly) with the case $\nu'(\RR_+) = c^{-1} \tilde\nu'(\RR_+)$, we 
use the fact that $\mu$ is continuous with respect to $\tilde\nu$ in the weak 
convergence topology (see \cite[Chap. 4]{zhang06}). 
By approximating $\tilde\nu$ by a sequence 
$\tilde\nu_k = \tilde\nu_k(\{0\}) + \tilde\nu'_k$ 
such that $\nu'(\RR_+) < c^{-1} \tilde\nu'_k$, we are led back to the first 
part of the proof. The result is obtained by continuity. 
\end{proof} 

\section{Density and support}  
\label{density-support} 

\subsection{Existence of a continuous density} 
\label{density-exists} 

This paragraph is devoted to establishing the following theorem: 
\begin{theorem}
\label{th-m(x)} 
For all $x \in \RR_* = \RR - \{ 0 \}$, the nontangential limit 
$\lim_{z \in \CC_+ \to x} m(z)$ exists. Denoting by $m(x)$ this limit, 
the function $\Im m(x)$ is continuous on $\RR_*$, and $\mu$ has
a continuous derivative $f(x) = \pi^{-1} \Im m(x)$ on $\RR_*$. \\
Similarly, the nontangential limits 
$\lim_{z \in \CC_+ \to x} \delta(z)$ and $\lim_{z \in \CC_+ \to x} \tdl(z)$ 
exist. Denoting respectively by $\Im\delta(x)$ and $\Im\tdl(x)$ these limits, 
the functions $\Im\delta(x)$ and $\Im\tdl(x)$ are both continuous on 
$\RR_*$, and both $\rho$ and $\tilde\rho$ have continuous derivatives 
on $\RR_+$. Finally $\support(\rho) \cap \RR_* = \support(\tilde\rho)
\cap \RR_* = \support(\mu) \cap \RR_*$. 
\end{theorem}

Since $\tilde\mu = c\mu + (1-c) \bs d_0$, it is obvious that we can replace 
$m$ with $\tilde m$ in the statement of the theorem.  

As soon as the existence of the three limits as $z \in \CC_+ \to x$ are
established, we know from the so called Stieltjes inversion formula that the 
densities exist (see \cite{sil-choi95}[Th. 2.1]). By a simple passage to the
limit argument (\cite[Th. 2.2]{sil-choi95}), we also know that these densities
are continuous. \\ 
To prove the theorem, we first prove that $\lim_{z \in \CC_+ \to x} \delta(z)$ 
and $\lim_{z \in \CC_+ \to x} \tdl(z)$ both exist for all $x \in \RR_*$ 
(Lemmas \ref{d2-bounded} to \ref{lm-cvg-delta}). 
This shows that both $\rho$ and $\tilde\rho$ have densities on $\RR_*$. 
Lemma \ref{lm-equiv-cvg} shows then that $\lim_{z \in \CC_+ \to x} m(z)$ 
exists, and furthermore, that the intersections of the supports of 
$\mu$, $\rho$ and $\tilde\rho$ with $\RR_*$ coincide.

\begin{lemma}
\label{d2-bounded}
$|\delta(z)|$ and $|\tdl(z)|$ are bounded on any bounded region of $\CC_+$ 
lying at a positive distance from the imaginary axis. 
\end{lemma} 
\begin{proof}
We first observe that for any $z \in \CC_+$, 
\begin{align*}
| \delta(z) | &\leq c \Bigl( 
\int \frac{t^2}{|z|^2 | 1 + \tdl(z) t |^2} \nu(dt) \Bigr)^{1/2} 
= \sqrt{c} \gamma(z,z^*)^{1/2} , \\
|\tdl(z) | &\leq \tilde\gamma(z,z^*)^{1/2}, 
\end{align*}
and we recall that $0 < |z|^2 \gamma(z,z^*) \tilde\gamma(z,z^*) < 1$. 
Using \eqref{m-delta}, we therefore get that 
$\sup_{z\in{\cal R}} | \tilde m(z) | < \infty$ where $\cal R$ is the region
alluded to in the statement of the lemma. \\
We now assume that $\sup_{z \in \cal R} | \tdl(z) | = \infty$ and raise a 
contradiction, the case where $\sup_{z \in \cal R} |\delta(z)|$ being treated 
similarly. By assumption, there exists a sequence $z_0, z_1, \ldots \in 
\cal R$ such that $| \tdl(z_k) | \to\infty$. By the inequalities above, we
get that $\tilde\gamma(z_k,z_k^*) \to \infty$, hence $\gamma(z_k,z_k^*) \to 0$ 
and therefore $\delta(z_k) \to 0$. In parallel, we have  
\begin{align*} 
z_0 \tilde m(z_0) - z_k \tilde m(z_k) &= 
\int\Bigl( \frac{-1}{1+\delta(z_0) t} + \frac{1}{1+\delta(z_k)t}\Bigr) 
\tilde\nu(dt) \\
&= (\delta(z_0) - \delta(z_k)) 
\int\frac{t}{(1+\delta(z_0) t)(1+\delta(z_k)t)} \tilde\nu(dt) . 
\end{align*} 
Using Identity \eqref{diff}, we obtain 
\[
(1 - z_0 z_k \gamma(z_0, z_k) \tilde\gamma(z_0, z_k) ) 
(\tdl(z_0) - \tdl(z_k)) \\ = 
(z_k^{-1} - z_0^{-1}) 
\frac{z_0 \tilde m(z_0) - z_k \tilde m(z_k)}{\delta(z_0) - \delta(z_k)} .
\]
By what precedes, $\sup_k | (z_k^{-1} - z_0^{-1}) 
(z_0 \tilde m(z_0) - z_k \tilde m(z_k)) | < \infty$. 
Moreover, $\liminf_{k} | \delta(z_0) - \delta(z_k) | > 0$ since 
$\Im\delta(z_0) > 0$. Cauchy-Schwarz inequality shows that 
$| \gamma(z_0, z_k) | \leq \gamma(z_0, z_0^*)^{1/2} \gamma(z_k, z_k^*)^{1/2}$
and $| \tilde\gamma(z_0, z_k) | \leq 
\tilde\gamma(z_0, z_0^*)^{1/2} \tilde\gamma(z_k, z_k^*)^{1/2}$. Therefore, 
\begin{align*}
\inf_k | 1 - z_0 z_k \gamma(z_0, z_k) \tilde\gamma(z_0, z_k) | &\geq 
1 - \sup_k | z_0 z_k \gamma(z_0, z_k) \tilde\gamma(z_0, z_k) | \\
&\geq 1 - (|z_0|^2 \gamma(z_0, z_0^*) \tilde\gamma(z_0, z_0^*))^{1/2} \times \\
& 
\ \ \ \ \ \ \ \ \ \ \ \ \ \ \ \ 
\ \ \ \ \ \ \ \ \ \ \ \
\sup_k (|z_k|^2 \gamma(z_k, z_k^*) \tilde\gamma(z_k, z_k^*))^{1/2} \\
&> 0 
\end{align*}  
which shows that $\sup_k | \tdl(z_k) | < \infty$. 
\end{proof}

\begin{lemma}
\label{m-borne} 
For $\ell = 1, 2$, the integrals 
\[
\int \frac{t^\ell}{| 1+ \tdl(z) t|^2} \nu(dt) 
\quad\text{and}\quad 
\int \frac{t^\ell}{|1+\delta(z)t |^2} \tilde\nu(dt) 
\]
are bounded on any bounded region $\mathcal R$ of $\CC_+$ lying at a positive 
distance from the imaginary axis.
\end{lemma}
\begin{proof}
We observe that for $\ell=2$, the integrals given in the statement of the
lemma are equal to 
$c^{-1} |z|^2 \gamma(z,z^*)$ and to $|z|^2 \tilde\gamma(z,z^*)$
respectively. We know that 
$\sup_{z\in{\mathcal R}}
 |z|^4 \gamma(z,z^*) \tilde\gamma(z,z^*) \leq \sup_{z\in{\mathcal R}}|z|^2 
< \infty$. Assume that $\tilde\gamma(z_n, z_n^*) \to\infty$ along some 
sequence $z_n \in{\mathcal R}$. 
Then $\gamma(z_n,z_n^*) \to 0$, which implies that the integrand of 
$|z_n|^2 \gamma(z_n,z_n^*)$ 
converges to zero $\nu$-almost everywhere. This implies in turn that 
$|\tdl(z_n)| \to\infty$ which contradicts Lemma \ref{d2-bounded}. The result
is proven for $\ell=2$. \\
We now consider the case $\ell=1$, focusing on the first integral that we write
as $\int_0^\infty t I(t)^{-1} \nu(dt)$. Since 
$\int_0^\infty t I(t)^{-1} \nu(dt) \leq \int_0^1 t I(t)^{-1} \nu(dt) +  
\int_1^\infty t^2 I(t)^{-1} \nu(dt)$, we only need to bound the first 
term at the right hand side. 
Denoting by $\1$ the indicator function, we have 
\begin{align*} 
\int_0^1 \frac{t}{I(t)} \nu(dt) &= 
\int_0^1 \frac{t}{I(t)} \1_{[0, |2\Re\tdl|^{-1}]}(t) \, \nu(dt) 
+ \int_0^1 \frac{t}{I(t)} \1_{(|2\Re\tdl|^{-1},\infty)}(t) \, \nu(dt) \\
&\leq 4 \int_0^1 t \nu(dt) + | 2 \Re\tdl | \int_0^\infty \frac{t^2}{I(t)} 
\nu(dt)
\end{align*} 
which is bounded. 
\end{proof} 

\begin{lemma} 
\label{lm-cvg-delta} 
For any $x \in \RR_*$, $\lim_{z \in \CC_+ \to x} \delta(z)$ and 
$\lim_{z \in \CC_+ \to x} \tdl(z)$ exist. 
\end{lemma} 
\begin{proof} 
If $\tilde\nu$ is a Dirac probability measure that we take without generality
loss as $\bs d_1$, then $\tdl(z) = \tilde m(z)$ converges as 
$z \in \CC_+ \to x$ to a non zero value \cite{sil-choi95}. Therefore, 
$\delta(z) = (- z\tdl(z))^{-1} - 1$ (see Eq.~\eqref{deltat}) also 
converges. We can therefore assume that neither $\nu$ nor $\tilde\nu$ is a 
Dirac measure. \\
We showed that $\delta$ and $\tilde \delta$ are bounded on any bounded region 
of $\CC_+$ lying away from the imaginary axis. 
Take two sequences $z_n$ and $\underline z_n$ in $\CC_+$ that converge to
the same $x \in \RR_*$, and such that $\tdl_n = \tdl(z_n)$ and 
$\underline\tdl_n = \tdl(\underline z_n)$
converge towards $\bs\tdl$ and $\underline{\bs\tdl}$ respectively, 
and $\delta_n = \delta(z_n)$ and $\underline\delta_n = \delta(\underline z_n)$
converge towards $\bs\delta$ and $\underline{\bs\delta}$ respectively. 
We shall show that $\bs\tdl = \underline{\bs\tdl}$ and 
$\bs\delta = \underline{\bs\delta}$. We start by writing 
\begin{multline*} 
(1 - z_n \underline z_n \gamma(z_n, \underline z_n) 
\tilde\gamma(z_n, \underline z_n) ) 
(\tdl_n - \underline\tdl_n) \\ = 
(z_n - \underline z_n) 
\int \frac{t}
{z_n \underline z_n( 1 + \delta_n t)
(1 + \underline\delta_n t)} \, \tilde\nu(dt) ,  
\end{multline*} 
and we have a similar equation controlling $\delta_n - \underline\delta_n$. 
The sequence of integrals at the right hand side is bounded by Cauchy-Schwarz
and by Lemma \ref{m-borne}. Therefore, the right hand side converges to 
zero as $z_n, \underline z_n \to x$.  We shall show that if 
$\bs\delta - \underline{\bs\delta} \neq 0$ or 
$\bs\tdl - \underline{\bs\tdl} \neq 0$, then 
$\liminf_n | 1 - z_n \underline z_n \gamma(z_n, \underline z_n) 
\tilde\gamma(z_n, \underline z_n) | > 0$, which raises a contradiction. \\
The real part of $z_n \underline z_n \gamma(z_n, \underline z_n) 
\tilde\gamma(z_n, \underline z_n)$ satisfies 
\begin{align*} 
&\Re ( z_n \underline z_n \gamma(z_n, \underline z_n) 
\tilde\gamma(z_n, \underline z_n) ) \\
&= 
\frac 14 \Bigl[ 
\int c \Bigl| \frac{ut}{z_n^*(1+\tdl_n^* t)(1+\delta_n^* u)} + 
\frac{ut}{\underline z_n (1+\underline\tdl_n t)(1+\underline\delta_n u)} 
\Bigr|^2 \nu(dt) \tilde\nu(du) \\ 
&\phantom{\frac 14 \Bigl[} 
- 
\int c \Bigl| \frac{ut}{z_n^*(1+\tdl_n^* t)(1+\delta_n^* u)} - 
\frac{ut}{\underline z_n (1+\underline\tdl_n t)(1+\underline\delta_n u)} 
\Bigr|^2 \nu(dt) \tilde\nu(du) \Bigr].
\end{align*} 
Writing concisely the right hand side as $(1/4) [ \chi_{1,n} - \chi_{2,n} ]$, 
we have $\chi_{1,n} / 4 < 1$ thanks to the inequalities 
$|z_n|^2 \gamma(z_n, z_n^*) \tilde\gamma(z_n, z_n^*) < 1$, 
$|\underline z_n|^2 \gamma(\underline z_n, \underline z_n^*) 
              \tilde\gamma(\underline z_n, \underline z_n^*) < 1$, and 
$| a+b|^2 \leq 2(|a|^2 + |b|^2)$. 
The term $\chi_{2,n}$ readily satisfies 
\begin{align*} 
\chi_{2,n} 
&\geq 
\int c t^2u^2 
\frac{\bigl|
\underline z_n (1+\underline\tdl_n t)(1+\underline\delta_n u) 
- z_n^*(1+\tdl_n^* t)(1+\delta_n^* u)  \bigr|^2} 
{| z_n \underline z_n|^2 (1+K t)^4 (1+ K u)^4} 
\nu(dt) \tilde\nu(du)   
\end{align*} 
where $K$ is a finite upper bound on the moduli of $\delta(z)$ and 
$\tdl(z)$ when $z\in\CC_+\to x$. 
Denoting the integrand at the right hand side as $F_n(t,u)$, we therefore get
that 
\begin{align*} 
| 1 - z_n \underline z_n \gamma(z_n, \underline z_n) 
\tilde\gamma(z_n, \underline z_n) | 
&\geq 
1 - \Re (z_n \underline z_n \gamma(z_n, \underline z_n) 
\tilde\gamma(z_n, \underline z_n) ) \\
& > \frac 14 \int F_n(t,u) \, \nu(dt) \tilde\nu(du) , 
\end{align*} 
hence 
\[
\liminf_n | 1 - z_n \underline z_n \gamma(z_n, \underline z_n) 
\tilde\gamma(z_n, \underline z_n) | \geq 
\frac 14 
\int \bs F(t,u) \, \nu(dt) \tilde\nu(du)  
\]
by Fatou's lemma, where 
\begin{align*}
\bs F(t,u) &= 
\frac{c u^2 t^2}{x^2 (1+Ku)^4 (1+Kt)^4} 
\Bigl|
\begin{bmatrix} 1 & u \end{bmatrix} \Bigl( 
\begin{bmatrix} 1 \\ \underline{\bs\delta}  \end{bmatrix} 
\begin{bmatrix} 1 & \underline{\bs\tdl}  \end{bmatrix} 
- 
\begin{bmatrix} 1 \\ \bs\delta^* \end{bmatrix} 
\begin{bmatrix} 1 & \bs\tdl^* \end{bmatrix} \Bigr) 
\begin{bmatrix} 1 \\ t \end{bmatrix}  
\Bigr|^2 \\
&= 
\frac{c}{x^2} 
\tr \Delta H(t) \Delta^* G(u) 
\end{align*} 
with 
\begin{gather*}
\Delta = \begin{bmatrix} 0 & \underline{\bs\tdl} - \bs\tdl^* \\
\underline{\bs\delta} - \bs\delta^* & 
\underline{\bs\delta} \underline{\bs\tdl} - \bs\delta^* \bs\tdl^* 
\end{bmatrix},  \\
H(t) = \frac{t^2}{(1+Kt)^4} \begin{bmatrix} 1 & t \\ t & t^2 \end{bmatrix},  
\quad \text{and} \quad  
G(u) = \frac{u^2}{(1+Ku)^4} \begin{bmatrix} 1 & u \\ u & u^2 \end{bmatrix} .
\end{gather*} 
Since $\nu$ is not a Dirac measure, 
\[
\Bigl( \int\frac{t^3}{(1+Kt)^4} \nu(dt) \Bigr)^2 
< 
\int\frac{t^2}{(1+Kt)^4} \nu(dt)  \ \times \ 
\int\frac{t^4}{(1+Kt)^4} \nu(dt) 
\]
therefore, the symmetric matrix $\int H(t) \, \nu(dt)$ is definite positive. 
For the same reason, the symmetric matrix $\int G(u) \, \tilde\nu(du)$ is also
definite positive. Observe now that 
$\bs\tdl \neq \underline{\bs\tdl} \Rightarrow 
\bs\tdl^* \neq \underline{\bs\tdl}$ and 
$\bs\delta \neq \underline{\bs\delta} \Rightarrow 
\bs\delta^* \neq \underline{\bs\delta}$ since the imaginary parts of 
$\bs\delta$, $\bs\tdl$, $\underline{\bs\delta}$ and $\underline{\bs\tdl}$
are non negative. Therefore, if $\bs\tdl \neq \underline{\bs\tdl}$ or 
$\bs\delta \neq \underline{\bs\delta}$, then the matrix $\Delta$ is 
non zero. It results that  
$\int \bs F(t,u) \, \nu(dt) \tilde\nu(du) > 0$ as desired.  
\end{proof} 

\begin{lemma} 
\label{lm-equiv-cvg}
For any $x \in \RR_*$, $\lim_{z \in \CC_+ \to x} m(z)$ exists. 
Let $m(x) = \lim_{z \in \CC_+ \to x} m(z)$,  
$\delta(x) = \lim_{z \in \CC_+ \to x} \delta(z)$ and 
$\tdl(x) = \lim_{z \in \CC_+ \to x} \tdl(z)$. Then 
\[
\Im \delta(x) > 0 \Leftrightarrow \Im \tdl(x) > 0 
\Leftrightarrow \Im m(x) > 0 . 
\]
\end{lemma} 
\begin{proof} 
The fact that $\lim_{z \in \CC_+ \to x} m(z)$ exists can be immediately deduced
from the first identity in \eqref{m-delta} and the previous lemma. 
Let us show that $\Im \delta(x) > 0 \Leftrightarrow \Im \tdl(x) > 0$. We have 
\[
\Im\tdl(z) = 
\frac{1}{|z|^2} \int \frac{\Im zt + \Im(z \delta(z)) t^2}
{|1 + \delta(z) t|^2} \tilde\nu(dt)
\]
Assume that $\lim_{z \in \CC_+ \to x} \Im \delta(z) = \Im \delta(x) > 0$. 
By Fatou's lemma, we get 
\[
\liminf_{z \in \CC_+ \to x} \Im\tdl(z) \geq 
\frac{1}{x^2} \int 
\frac{x \Im\delta(x) t^2}{(1 + \Re\delta(x) t)^2 + t^2 (\Im\delta(x))^2} 
\tilde\nu(dt) > 0 . 
\]
Using this same argument with the roles of $\delta$ and $\tdl$ 
interchanged, we get that $\Im \delta(x) > 0 \Leftrightarrow 
\Im \tdl(x) > 0$. \\
Using \eqref{ts} and Fatou's lemma again, we also obtain that 
$\Im \tdl(x) > 0 \Rightarrow \Im m(x) > 0$. Conversely, 
$\Im m(x) = -c^{-1} \Im(\delta(x)\tdl(x)) = - c^{-1} ( \Re\delta(x) \Im\tdl(x) 
+ \Im\delta(x) \Re\tdl(x) )$. Therefore, 
$\Im m(x) > 0 \Rightarrow (\Im\delta(x) > 0 \ \text{or} \ \Im\tdl(x) > 0)  
\Leftrightarrow \Im\tdl(x) > 0$. 
\end{proof} 

\subsection{Determination of $\support(\mu)$} 
\label{determine-support}

In the remainder, we characterize $\support(\mu) \cap \RR_* = 
\support(\tilde\rho) \cap \RR_*$, focusing on the measure $\tilde\rho$. 
In the following, we let 
\[
\cal D = \left\{\begin{array}{l} 
\{0 \} \cup 
\{ \bs\delta \in \RR_* \, : \, - \bs\delta^{-1} \not\in \support(\tilde\nu) \} 
\ \text{if} \ \support(\tilde\nu) \ \text{is compact}, \\ 
\{ \bs\delta \in \RR_* \, : \, - \bs\delta^{-1} \not\in \support(\tilde\nu) \} 
\ \text{otherwise},  
\end{array}\right. 
\]
and 
\[
\widetilde{\cal D} = \left\{\begin{array}{l} 
\{0 \} \cup 
\{ \bs\tdl \in \RR_* \, : \, - \bs\tdl^{-1} \not\in \support(\nu) \} 
\ \text{if} \ \support(\nu) \ \text{is compact}, \\ 
\{ \bs\tdl \in \RR_* \, : \, - \bs\tdl^{-1} \not\in \support(\nu) \} 
\ \text{otherwise}.  
\end{array}\right. 
\]
Notice that $\cal D$ and $\widetilde{\cal D}$ are both open.
 
\begin{proposition}
\label{direct} 
If $\bs x \in \RR_*$ does not belong to $\support(\mu)$, then 
$\delta(\bs x) \in \cal D$, $\tilde\delta(\bs x) \in \widetilde{\cal D}$, and  
$1 - \bs x^2 \gamma(\bs x,\bs x) \tilde\gamma(\bs x,\bs x) > 0$.
\end{proposition}
\begin{proof}
Since $\support(\mu) \cap \RR_* = \support(\rho) \cap \RR_* = 
\support(\tilde\rho) \cap \RR_*$ and 
since the Stieltjes Transform of a positive measure is real and increasing
on the real axis outside the support of this measure, $\delta(\bs x) \in \RR, 
\ 
\tilde\delta(\bs x) \in \RR$ and $\tilde\delta'(\bs x) > 0$. 
Extending Equation \eqref{diff} to a neighborhood of $\bs x$, we get 
\[
\tilde\delta'(\bs x) = 
\frac{1}{1 - \bs x^2 \gamma(\bs x, \bs x) \tilde\gamma(\bs x, \bs x)} 
\int \frac{t}
{\bs x^2( 1 + \delta(\bs x) t)^2} \, \tilde\nu(dt)  
\]
hence $1 - \bs x^2 \gamma(\bs x, \bs x) \tilde\gamma(\bs x, \bs x) > 0$. \\
We now show that $\delta(\bs x) \in \cal D$. Assume $\delta(\bs x) \neq 0$.
Denoting by $m_{\tilde\nu}$ the Stieltjes Transform of $\tilde\nu$, Equation 
\eqref{deltat} can be rewritten as $m_{\tilde\nu}(-\delta(z)^{-1}) = 
\delta(z) + z \delta^2(z) \tilde\delta(z)$. 
Making $z$ converge from $\CC_+$ to a point $x$ lying in a small neighborhood 
of $\bs x$ in $\RR$, the right hand side of this equation converges to a real 
number, and $-\delta(z)^{-1}$ converges from 
$\CC_+$ to a point in a neighborhood of $-\delta(\bs x)^{-1}$ in $\RR$. Since 
$m_{\tilde\nu}$ is real on this neighborhood, the load of this neighborhood by 
$\tilde\nu$ is zero, which implies that $\delta(\bs x) \in \cal D$. 
Assume now that $\delta(\bs x) = 0$. Then there exists
$\bs x_0 \not\in \support(\rho)$ such that $\bs x_0 < \bs x$ and $\delta(x)$
increases from $\delta(\bs x_0)$ to zero on $[\bs x_0, \bs x]$. The argument
above shows that $\tilde\nu([-\delta^{-1}(\bs x_0), -\delta^{-1}(x)]) = 0$ for 
any $x \in [ \bs x_0, \bs x)$. Making $x \uparrow \bs x$, we obtain that 
$\tilde\nu([-\delta^{-1}(\bs x_0), \infty)) = 0$, in other words, $\tilde\nu$ 
is compactly supported. It results that $\delta(\bs x) \in {\cal D}$. 
The same argument shows that $\tdl(\bs x) \in \widetilde{\cal D}$. 
\end{proof}

\begin{proposition}
\label{recip} 
Given $\bs\tdl \in \widetilde{\cal D}$, assume there exists  
$\bs x \in \RR_*$ for which 
\begin{equation} 
\label{reciproque} 
\begin{split} 
&\displaystyle{\bs\delta = c \int \frac{t}
{-\bs x(1+\bs\tdl t)} \nu(dt) \ \in {\cal D}}, 
\\ 
&\displaystyle{\bs\tdl = \int \frac{t}{-\bs x(1+\bs\delta t)} \tilde\nu(dt)} , 
\end{split}
\end{equation} 
and 
\begin{equation}
\label{deriv>0} 
1 - 
\bs x^2 \bs\gamma(\bs x,\bs{\tilde\delta}) \bs{\tilde\gamma}(\bs x,\bs\delta) 
> 0 
\end{equation} 
where 
\begin{align*}
\bs\gamma(\bs x, \bs\tdl) &= 
c \int \frac{t^2}
{\bs x^2(1+\bs\tdl t)^2} \nu(dt), \quad \text{and} \\ 
\bs{\tilde\gamma}(\bs x, \bs\delta) &= 
\int \frac{t^2}
{\bs x^2(1+\bs\delta t)^2} \tilde\nu(dt)   . 
\end{align*}
Then $\bs x \not\in \support(\mu)$. 
\end{proposition} 
\begin{proof}
Let $(\bs\tdl,\bs x)$ be a solution of Equations \eqref{reciproque} such
that $\bs\tdl \in \widetilde{\mathcal D}$, $\bs\delta \in {\mathcal D}$, and
Inequality \eqref{deriv>0} is satisfied. 
Define on a small enough open neighborhood of $(\bs\tdl,\bs x)$ in $\RR^2$ 
the function 
\begin{equation}
\label{F-real} 
\bs F(\tdl,x) = 
\int\frac{t} {\displaystyle{-x + 
ct \int \frac{u}{1+u\tilde\delta} \nu(du)} } 
\tilde\nu(dt) - \tdl  . 
\end{equation} 
Clearly, $\bs F(\bs\tdl,\bs x) = 0$, and a small calculation shows that 
\[
\frac{\partial \bs F}{\partial\tdl} (\bs\tdl,\bs x) = 
-1+\bs x^2 \bs\gamma(\bs x,\bs{\tilde\delta}) \bs{\tilde\gamma}(\bs x,\bs\delta)
< 0 
\]
(in this calculation, integration and differentiation can be exchanged since
$\bs\tdl \in \widetilde{\mathcal D}$ and $\bs\delta \in {\mathcal D}$). 
By the implicit function theorem, there is a real function $\underline\tdl(x)$ 
defined on a real neighborhood $V$ of $\bs x$ such that 
$\underline\tdl(\bs x) = \bs\tdl$ and every couple $(x,\underline\tdl(x))$ 
for $x \in V$ satisfies the assumptions of the statement of the proposition.
To establish the proposition, it will be enough to show that for any $x \in V$, 
$\underline\tdl(x) = \lim_{z\in\CC_+ \to x} \tdl(z)$. \\ 
Fixing $x\in V$, it is easy to see that for any $z \in \CC_+$,  
\begin{equation}
\label{GtG} 
(1 - z x \Gamma(z, x) 
\widetilde\Gamma(z, x) ) 
(\tdl(z) - \underline\tdl(x)) = (z-x)
\int \frac{t}
{z x( 1 + \delta(z) t)
(1 + \underline\delta(x) t)} \, \tilde\nu(dt)   
\end{equation} 
where $\underline\delta(x) = - c x^{-1} \int t(1+\underline\tdl(x) t)^{-1} 
\nu(dt)$, 
\begin{align*}
\Gamma(z, x) &= 
c \int \frac{t^2}
{z x(1+\tdl(z)t)(1+\underline\tdl(x)t)} \nu(dt), \quad \text{and} \\ 
\widetilde\Gamma(z, x) &= 
\int \frac{t^2}
{z x(1+\delta(z)t)(1+\underline\delta(x)t)} \tilde\nu(dt) . 
\end{align*}
By the Cauchy-Schwarz inequality, Lemma \ref{m-borne} and the fact that 
$\bs\delta \in {\cal D}$, the integral at the right hand side of~\eqref{GtG}
remains bounded as $z\to x$. Repeating the derivations made in the proof of 
Lemma~\ref{lm-cvg-delta} (the case where $\nu$ or $\tilde\nu$ is a Dirac 
measure being dealt with as in~\cite{sil-choi95}), we can show that 
$\tdl(x) = \underline\tdl(x)$. 
\end{proof} 

\subsection{Practical procedure for determining $\support(\mu)$} 
\label{practical-support}

Proposition \ref{direct} shows that for any 
$\bs x \in \support(\mu)^c \cap \RR_*$, there exists a couple 
$(\bs\delta,\bs\tdl)$ that satisfies the assumptions of Proposition 
\ref{recip}. The reverse is shown by Proposition \ref{recip}. \\
These observations suggest a practical procedure for determining the support
of $\mu$. We let $\bs\tdl$ run through $\widetilde{\mathcal D}$. For every one 
of these $\bs\tdl$, we compute 
\[
\psi(\bs\tdl) = c \int \frac{t} {1+\bs\tdl t} \nu(dt)
\]
then we find numerically the solutions of the equation in $\bs x$ 
\[
\bs\tdl = \int \frac{t}{-\bs x+ \psi(\bs\tdl) t} \tilde\nu(dt) . 
\]
for which $-\bs x^{-1} \psi(\bs\tdl) \in \mathcal D$. 
Among these solutions, we retain those points $\bs x$ for which 
\[
 1 - 
 c \int \frac{t^2} {(1+\bs\tdl t)^2} \nu(dt) 
 \int \frac{t^2}{(\bs x - \psi(\bs\tdl) t)^2} \tilde\nu(dt) 
 > 0  .
\] 
What is left after making $\bs\tdl$ run through $\widetilde{\mathcal D}$ is 
$\support(\mu) \cap \RR_*$. 
The figure gives an idea of the result. 

\begin{figure}[h] 
\centering
\includegraphics[width=0.8\linewidth]{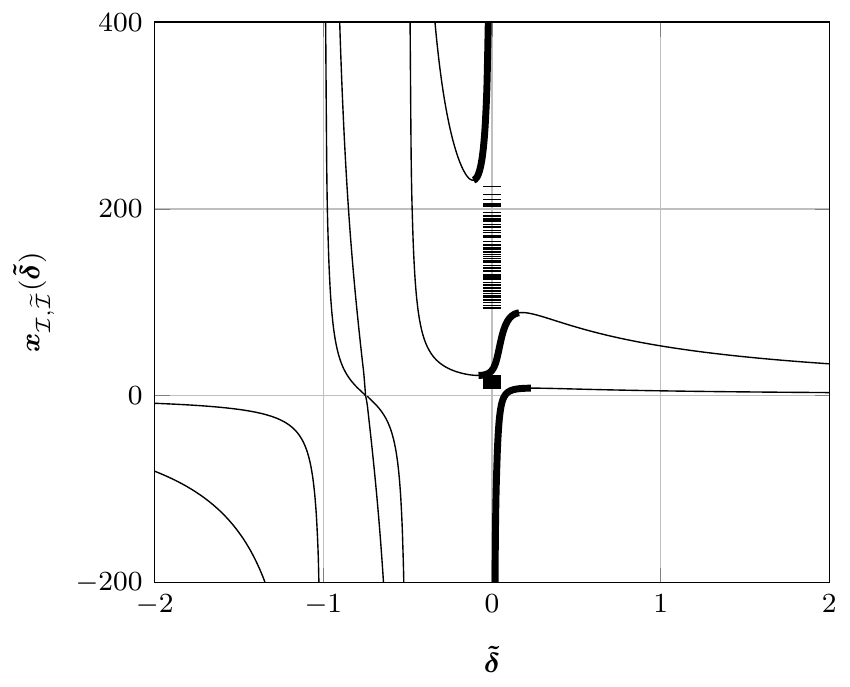}
\caption{${\bs x}_{\mathcal I,\widetilde{\mathcal I}}({\bs\tdl})$ for each component pairs $\mathcal I$ of $\mathcal D$ and $\widetilde{\mathcal I}$ of $\widetilde{\mathcal D}$. In thick line, positions for which $1-{\bs x}^2{\bs\gamma}({\bs x},{\bs\tdl}){\bs{\tilde{\gamma}}}({\bs x},{\bs\delta})>0$. On the vertical axis, in black dashes, empirical eigenvalue positions for $N=1000$. Setting: $c=10$, $\nu=1/2(\delta_1+\delta_2)$, $\tilde{\nu}=1/2(\delta_1+\delta_{10})$.}
\label{fig:x_I}
\end{figure}

\subsection{Properties of the graph of $\bs x$ versus $\bs\tdl$ 
and the consequences}
\label{graph-properties}

The two following propositions will help us bring out some of the
properties of the graph of $\bs x$ versus $\bs\tdl$. In their statements, we 
assume that the triples $(\bs\tdl_1, \bs\delta_1, \bs x_1)$ and 
$(\bs\tdl_2, \bs\delta_2, \bs x_2)$ satisfy both the statement of Proposition
\ref{recip}. 

\begin{lemma}
\label{disjoint}
$\bs\tdl_1 \neq \bs\tdl_2 \Rightarrow \bs x_1 \neq \bs x_2$ and 
$\bs\delta_1 \neq \bs\delta_2 \Rightarrow \bs x_1 \neq \bs x_2$. 
\end{lemma}
\begin{proof}
We know that $\bs\tdl_i = \lim_{z\in\CC_+\to \bs x_i} \tdl(z)$ for $i=1,2$. 
Assume that $\bs\tdl_1 \neq \bs\tdl_2$. Then having $\bs x_1 = \bs x_2$ would 
violate this convergence. 
\end{proof}


\begin{lemma}
\label{increase} 
If $\bs\tdl_1 < \bs\tdl_2$, if $\bs x_1 \bs x_2 > 0$, and if 
$[\bs\delta_1 \wedge \bs\delta_2, \bs\delta_1 \vee \bs\delta_2 ] 
\subset {\mathcal D}$, then $\bs x_1 < \bs x_2$. 
\end{lemma} 


\begin{proof}
We use the identity 
\begin{multline*} 
\Bigl(1 - \bs x_1 \bs x_2 \gamma(\bs x_1, \bs x_2) 
\tilde\gamma(\bs x_1, \bs x_2) \Bigr) 
(\bs\tdl_1 - \bs\tdl_2) \\ = 
(\bs x_1 - \bs x_2) 
\int \frac{t}
{\bs x_1 \bs x_2( 1 + \bs\delta_1 t)(1 + \bs\delta_2 t)} \, \tilde\nu(dt) ,  
\end{multline*}
see \eqref{diff}. 
By the Cauchy-Schwarz inequality, 
$1 - \bs x_1 \bs x_2 \gamma(\bs x_1, \bs x_2) \tilde\gamma(\bs x_1, \bs x_2)
> 0$. 
Let us show that the integral $I$ at the right hand side of the equation 
above is positive. 
Assume that for some $t \in\support(\tilde\nu)$, the numbers 
$1 + \bs\delta_1 t$ and $1 + \bs\delta_2 t$ do not have the same sign. Then
there exists $\bs\delta \in ( \bs\delta_1 \wedge \bs\delta_2, 
\bs\delta_1 \vee \bs\delta_2)$ such that $1 + \bs\delta t = 0$. But this
contradicts 
$[\bs\delta_1 \wedge \bs\delta_2, \bs\delta_1 \vee \bs\delta_2 ] 
\subset {\mathcal D}$. Hence $I > 0$, which shows that $\bs x_1 - \bs x_2$ and 
$\bs\tdl_1 - \bs\tdl_2$ have the same sign. 
\end{proof} 

In order to better understand the incidence of these propositions, let us 
describe more formally the procedure for determining the support of $\mu$. 
Equations \eqref{reciproque} can be rewritten as 
$- \bs x \bs\delta \bs\tdl = g(\bs\tdl) = \tilde g(\bs\delta)$ where
\[
g(\bs\tdl) = c \int \frac{\bs\tdl t} {1+\bs\tdl t} \nu(dt) \quad \text{and} 
\quad 
\tilde g(\bs\delta) = 
\int \frac{\bs\delta t} {1+\bs\delta t} \tilde\nu(dt) 
\]
are both increasing on any interval of $\widetilde{\cal D}$ and $\cal D$
respectively. 
Let $\mathcal I$ and $\widetilde{\mathcal I}$ be two connected components 
of $\mathcal D$ and $\widetilde{\mathcal D}$ respectively\footnote{To give 
an example,
assume that $\support(\nu) \cap \RR_* = [a_1, b_1] \cup [a_2, b_2] \cup 
\cdots \cup [a_K, b_K ]$ where $0 < a_1 \leq b_1 < a_2 
\leq b_2 < \cdots < a_K \leq b_K < \infty$. Then the connected components of
$\widetilde{\cal D}$ are $(-\infty, -a_1^{-1}), (-b_1^{-1}, -a_2^{-1}), 
\ldots, (-b_{K-1}^{-1}, -a_{K-1}^{-1})$, and $(-b_{K}^{-1}, \infty)$.}. 
Assume that 
$\tilde g({\mathcal I}) \cap g(\widetilde{\mathcal I}) \neq \emptyset$. 
Since $\tilde g$ is increasing, it has a local inverse
$\tilde g^{-1}_{\mathcal I, \widetilde{\mathcal I}}$ on 
$g(\widetilde{\mathcal I})$. Let $\bs\delta = 
\tilde g^{-1}_{\mathcal I, \widetilde{\mathcal I}} \circ g(\bs\tdl)$ 
and consider the function 
\begin{equation}
\label{eq-x} 
\bs x_{\mathcal I, \widetilde{\mathcal I}}(\bs\tdl) = 
- \frac{g(\bs\tdl)}{\bs\delta \bs\tdl} = 
- \frac{g(\bs\tdl)}{\bs\tdl \times 
\tilde g^{-1}_{\mathcal I, \widetilde{\mathcal I}} \circ g(\bs\tdl)} , 
\end{equation} 
with domain the open set 
$\dom(\xII) = \{ \bs\tdl \in \widetilde{\cal I} \, : \, \exists \bs\delta 
\in {\cal I} \ \text{such that} \ \tilde g(\bs\delta) = g(\bs\tdl) \ 
\text{and} \ \bs\delta \neq 0 \}$. 
Computing $\bs x_{\mathcal I, \widetilde{\mathcal I}}(\bs\tdl)$ on all 
connected components $\cal I$ and $\widetilde{\cal I}$ and dropping the values 
of $\bs x$ for which 
$1 - 
\bs x^2 \bs\gamma(\bs x,\bs{\tilde\delta}) \bs{\tilde\gamma}(\bs x,\bs\delta) 
> 0$, we are of course left with $\support(\mu) \cap \RR_*$.  \\ 
Thanks to Lemmas \ref{disjoint}-\ref{increase}, the functions 
$\bs x_{\mathcal I, \widetilde{\mathcal I}}$ have the following properties:
\begin{enumerate} 
\item\label{one-branch} 
For any $\bs x_0 \in \RR_*$, at most one function 
$\bs x_{\mathcal I, \widetilde{\mathcal I}}$ satisfies 
$\bs x_{\mathcal I, \widetilde{\mathcal I}}(\bs\tdl) = \bs x_0$ and 
$\bs x'_{\mathcal I, \widetilde{\mathcal I}}(\bs\tdl) > 0$ 
by Lemma \ref{disjoint}. \\
Note that more than one function $\bs x_{\mathcal I, \widetilde{\mathcal I}}$ 
can be possibly increasing at a given $\bs\tdl \in \widetilde{\mathcal D}$, as 
the figure shows. 

\item\label{increase-once} 
We show below that there is exactly one couple $({\cal I}, \widetilde{\cal I})$
for which $\bs x_{\mathcal I, \widetilde{\mathcal I}}$ has negative values
and is increasing from $-\infty$ to zero where it is negative. 
Moreover, for any couple $({\cal I}, \widetilde{\cal I})$ and for any 
$[ \bs\tdl_1 , \bs \tdl_2 ] \in \widetilde{\mathcal I}$ such that
$\bs x_{\mathcal I, \widetilde{\mathcal I}}(\bs\tdl_i) > 0$ and 
$\bs x'_{\mathcal I, \widetilde{\mathcal I}}(\bs\tdl_i) > 0$, $i=1,2$, 
the function  $\bs x_{\mathcal I, \widetilde{\mathcal I}}(\tdl)$ never 
decreases between $\tdl_1$ and $\tdl_2$ by Lemma \ref{increase}. \\
In summary, if a branch of a 
$\bs x_{\mathcal I, \widetilde{\mathcal I}}(\tdl)$ is increasing at two points 
$\tdl_1$ and $\tdl_2$, then it never decreases between these two points.

\item\label{les-bords} 
Let $b = \sup(\support(\nu)) \in (0,\infty]$ and 
$\tilde b = \sup(\support(\tilde\nu)) \in (0, \infty]$, and let us study the
behavior of $\bs x_{\mathcal I, \widetilde{\mathcal I}}$ when 
$\widetilde{\cal I} = (-b^{-1}, \infty)$ and 
${\cal I} = (-\tilde b^{-1}, \infty)$. Assume $b = \tilde b = \infty$. 
By the fact that the functions $\delta(x)$ and $\tdl(x)$ are both positive 
and increasing on $(-\infty, 0)$ and by Lemma \ref{disjoint}, 
the branch $\bs x_{\mathcal I, \widetilde{\mathcal I}}(\bs\tdl)$ is increasing 
where it is negative, it is the only branch having this property, and 
$\bs x_{\mathcal I, \widetilde{\mathcal I}}(\bs\tdl) \to -\infty$ as 
$\bs\tdl \downarrow 0$.  \\
Assume now that $b = \infty$ and $\tilde b < \infty$. Here it is easy to 
notice that $g((-\tilde b^{-1},0)) \cap \tilde g((0,\infty)) = \emptyset$ 
which implies that we can replace ${\cal I}$ with $(0,\infty)$. 
As in the former case, the graph of 
$\bs x_{\mathcal I, \widetilde{\mathcal I}}$ consists in one branch that has
the same properties as regards the negative values of $\bs x$. 
The same conclusion holds when $b < \infty$ and $\tilde b = \infty$. \\
Finally, assume that $b, \tilde b < \infty$. Here 
$g(\bs\tdl)/\bs\tdl \approx C$ and $\bs\delta \approx C' \bs\tdl$ near zero, 
where $C, C' > 0$. Consequently, the graph of 
$\bs x_{\mathcal I, \widetilde{\mathcal I}}(\bs\tdl)$ consists in two
branches, one on $(-b^{-1}, 0)$ and one on $(0,\infty)$. The first branch
converges to infinity as $\bs\tdl \uparrow 0$, showing that $\mu$ is 
compactly supported, and the second branch behaves below zero as its analogues 
above. These two branches appear on the figure.


\item\label{no-use}
Assume that $a = \inf(\support(\nu) \cap \RR_*) > 0$ and let 
$\widetilde{\cal I} = (-\infty, - a^{-1})$. Then $g(\bs\tdl)$ increases from 
$c$ as $\bs\tdl$ increases from $-\infty$. If $\bs\delta < 0$, then 
$\bs x_{\mathcal I, \widetilde{\mathcal I}}(\bs\tdl) < 0$ since
$g(\bs\tdl)/\bs\tdl < 0$, and the conclusions of Item \eqref{les-bords} show 
that the branches $\bs x_{\mathcal I, \widetilde{\mathcal I}}$ need not be 
considered for determining $\support(\mu)$ when ${\cal I}\subset (-\infty, 0)$. 
It remains to study $\bs x_{\mathcal I, \widetilde{\mathcal I}}$ for 
${\cal I} = (-\tilde b^{-1}, \infty)$. On $(0, \infty)$, the
function $\tilde g(\bs\delta)$ increases from $0$ to $1$, hence $\tilde
g((0, \infty)) \cap g(\widetilde{\cal I}) \neq \emptyset$ if and
only if $c < 1$. In that case, it can be checked that $\bs x_{\mathcal I,
\widetilde{\mathcal I}}(\bs\tdl)$ increases from $0$ as $\bs\tdl$ increases
from $-\infty$. In conclusion, if $a > 0$ and $c < 1$, then 
$\inf(\support(\mu) \cap \RR_*) > 0$, and the location of this infimum is 
provided by the branch $\bs x_{\mathcal I, \widetilde{\mathcal I}}$. \\ 
Similarly, if $\tilde a = \inf(\support(\tilde\nu) \cap \RR_*) > 0$, 
${\cal I} = (-\infty, -\tilde a^{-1})$ and 
$\widetilde{\cal I} \subset (-\infty, 0)$, then the branches 
$\bs x_{\mathcal I, \widetilde{\mathcal I}}$ need not be considered. If in 
addition $c > 1$, then $\inf(\support(\mu) \cap \RR_*) > 0$, and the location 
of this infimum is provided by the branch 
$\bs x_{\mathcal I, \widetilde{\mathcal I}}$ for 
${\cal I} = (-\infty, -\tilde a^{-1})$ and 
$\widetilde{\cal I} = (-b^{-1}, \infty)$.

\end{enumerate} 

We terminate this paragraph with the following two results:

\begin{proposition}
\label{bulks} 
Assume that $\support(\nu) \cap \RR_*$ and $\support(\tilde\nu) \cap \RR_*$  
consist in $K$ and $\widetilde K$ connected components respectively. Then 
$\support(\mu) \cap \RR_*$ consists in at most $K \widetilde K$ connected 
components. 
\end{proposition}
\begin{proof}
When $\nu$ is compactly supported, 
$\support(\nu - \nu(\{0\}) \bs d_0) = [a_1, b_1] \cup [a_2, b_2] \cup 
\cdots \cup [a_K, b_K ]$ where 
$0 < a_1 \leq b_1 < a_2 \leq b_2 < \cdots < a_K \leq b_K < \infty$ 
or 
$0 = a_1 < b_1 < a_2 \leq b_2 < \cdots < a_K \leq b_K < \infty$.  
In the first case, the connected components of $\widetilde{\cal D}$ are 
$\widetilde{\cal I}_0 = (-\infty, - a_1^{-1})$, 
$\widetilde{\cal I}_1 = (- b_1^{-1}, - a_2^{-1}),\ldots, 
\widetilde{\cal I}_K = (- b_K^{-1}, \infty)$. In the second case, these
connected components are $\widetilde{\cal I}_1,\ldots, \widetilde{\cal I}_K$. 
If $\nu$ is not compactly supported, $a_K < b_K = \infty$ and the expressions 
of the connected components of $\widetilde{\cal D}$ are unchanged. 
With similar notations, the connected components of $\cal D$ are 
${\cal I}_0,\ldots, {\cal I}_{\widetilde K}$ or ${\cal I}_1,\ldots, 
{\cal I}_{\widetilde K}$ according to whether 
$\inf(\support(\tilde\nu) \cap \RR_*)$ is positive or not. 
Let $s = \inf(\support(\mu) \cap \RR_*)$ and $S = \sup(\support(\mu))$. 
Following the observations we just made, we notice that the only possible
$\bs x_{{\cal I}_k, \widetilde{\cal I}_{\tilde k}}(\bs\tdl) \in (s,S)$ such
that 
$\bs x_{{\cal I}_k, \widetilde{\cal I}_{\tilde k}}'(\bs\tdl) > 0$ 
are those for which $1\leq k \leq K$, $1\leq \tilde k\leq \widetilde K$, and 
$(k, \tilde k) \neq (K, \widetilde K)$. Therefore, the number of intervals
of $\support(\mu)^c \cap (s,S)$ is upper bounded by $K \widetilde K - 1$,
hence the result. 
\end{proof} 

\begin{proposition}
\label{compactness}
$\support(\mu)$ is compact if and only if $\support(\nu)$ and 
$\support(\tilde\nu)$ are compact.
\end{proposition}
\begin{proof}
The ``if'' part has been shown by Item \eqref{les-bords} above. Assume 
$\support(\mu)$ is compact. 
The fact that $\support(\rho) \cap \RR_* = \support(\tilde\rho)
\cap \RR_* = \support(\mu) \cap \RR_*$ and the equation 
$m_{\tilde\nu}(-\delta(z)^{-1}) = \delta(z) + z \delta^2(z) \tilde\delta(z)$
show that $m_{\tilde\nu}(z)$ can be analytically extended to $(A,\infty)$ 
for $A$ large enough, hence the compactness of $\support(\tilde\nu)$. 
A similar conclusion holds for $\support(\nu)$. 
\end{proof}

\subsection{Properties of the density of $\mu$ on $\RR_*$} 
\label{properties-density}

\begin{theorem} 
\label{holom}
The density $f(x)$ specified in the statement of Theorem \ref{th-m(x)} is
analytic for every $x \neq 0$ for which $f(x) > 0$. 
\end{theorem} 
\begin{proof}
We can assume that $\nu$ is not a Dirac measure, otherwise 
see~\cite{sil-choi95}. Let $x_0 \neq 0$ be such that $f(x_0) > 0$. 
We start by showing that $\tdl(z)$ can be analytically extended from 
$\CC_+$ to a neighborhood of $x_0$ in $\CC$. 
Write 
\begin{gather*} 
\gamma(x_0, x_0) = \lim_{z\in\CC_+\to x_0} \gamma(z,z), \quad 
\tilde\gamma(x_0, x_0) = \lim_{z\in\CC_+\to x_0} \tilde\gamma(z,z), \\ 
\Gamma(x_0, x_0) = \lim_{z\in\CC_+\to x_0} \gamma(z,z^*), \quad 
\widetilde\Gamma(x_0, x_0) = \lim_{z\in\CC_+\to x_0} \tilde\gamma(z,z^*). 
\end{gather*} 
Making $z \in \CC_+$ converge to $x_0$ in Equation \eqref{gamma(z,z*)} and 
recalling that the integral at the right hand side of this equation remains 
bounded and that $\Im\tdl(x_0) > 0$, we get that 
$x_0^2 \Gamma(x_0, x_0) \widetilde\Gamma(x_0, x_0) = 1$. 
Any integrable random variable $X$ satisfies $| \EE X | \leq \EE|X|$, the 
equality being achieved if and only if $X = \theta |X|$ almost everywhere, 
where $\theta$ is a modulus one constant. Consequently, 
$|\gamma(x_0, x_0)| < \Gamma(x_0, x_0)$ since $\nu$ is not a Dirac measure, 
and $|\tilde\gamma(x_0, x_0)| \leq \widetilde\Gamma(x_0, x_0)$. Therefore, 
$|x_0^2 \gamma(x_0, x_0) \tilde\gamma(x_0, x_0)| < 1$. 
Now, since $\Im\tdl(x_0) > 0$, it is easy to see by inspecting Equation 
\eqref{tdelta-seul} that the function $F(\tdl,z)$ 
which is holomorphic on $\CC_+^2$ can be analytically extended to a 
neighborhood of $(\tdl(x_0), x_0)$ in $\CC_+ \times \CC_*$ where 
$\CC_* =\CC - \{ 0 \}$. Observing that 
\[
\frac{\partial F}{\partial\tdl} (\tdl(x_0), x_0) = 
-1+ x_0^2 \gamma(x_0,x_0) \tilde\gamma(x_0, x_0) \neq 0 
\]
and invoking the holomorphic implicit function theorem,  
we get that there exists a  
neighborhood $V \subset \CC_*$ of $x_0$, a neighborhood $V' \subset \CC_+$
of $\tdl(x_0)$ and a holomorphic function $\underline\tdl : V \to V'$ such that
\[
\{ (z,\tdl) \in V \times V' \, : \, F(\tdl, z) = 0 \} 
\ = \ 
\{ (z, \underline\tdl(z)) \, : \, z \in V \} .
\] 
Since $\tdl(z)$ and $\underline\tdl(z)$ coincide on $V \cap \CC_+$, the 
function $\underline\tdl(z)$ is an analytic extension of $\tdl(z)$ on $V$. \\
This result shows in conjunction with Equation \eqref{ts} that $m(z)$ can be 
extended analytically to $V$. Therefore, writing 
$m(z) = \sum_{\ell \geq 0} a_\ell (z - x_0)^\ell$ we get that 
$f(x) = \pi^{-1} \sum_{\ell\geq 0} \Im a_\ell \, (x - x_0)^\ell$ near
$x_0$. 
\end{proof}

We now study the behavior of the density $f(x)$ near a boundary point $a > 0$ 
of $\support(\mu)$. The observations made above show that when $a$ is a left
end point (resp.~a right end point) of $\support(\mu)$, it is a local supremum 
(resp.~a local infimum) of one of the functions $\xII$. 
Parallelling the assumptions made in \cite{MarPas67}, \cite{sil-choi95} and 
\cite{DozSil07b}, we restrict ourselves to the case where $a= \xII(\bs\tdl_a)$ 
for some $\bs\tdl_a \in \dom(\xII)$. In that case, $\xII$ is of course 
analytical around $\bs\tdl_a$ and $\xII'(\bs\tdl_a) = 0$. \\
Note that this assumption might not be satisfied for some choices of the 
measures $\nu$ and $\tilde\nu$. Assuming $a > 0$ is a left end point of 
$\support(\mu)$, it is for instance possible that the function $\xII(\bs\tdl)$ 
increases to $a$ as $\bs\tdl \uparrow \bs\tdl_a$ with 
$-\bs\tdl_a^{-1} \in \partial\nu$. We however note that our assumption is 
valid when the measures $\nu$ and $\tilde\nu$ are both discrete. 

\begin{theorem}
\label{sq-root}
Let $\cal I$ and $\widetilde{\cal I}$ be two connected components of 
$\cal D$ and $\widetilde{\cal D}$ respectively, and assume that 
$\xII$ reaches a maximum at a point $\bs\tdl_a \in \dom(\xII)$. Then 
$\xII''(\bs\tdl_a) < 0$. Furthermore, for $\varepsilon > 0$ small enough, 
$f(x) = H(\sqrt{x-a})$ on $(a, a+\varepsilon)$ where $H(x)$ is a real 
analytical function near zero, $H(0) = 0$, and 
\[
H'(0) = \frac{1}{\pi a} \sqrt{\frac{-2}
{\xII''(\bs\tdl_a)}} \int \frac{t}{(1+\bs\tdl_a t)^2} \nu(dt) .
\]
Assume now that $\xII$ reaches a minimum at a point $\bs\tdl_a \in \dom(\xII)$.
Then $\xII''(\bs\tdl_a) > 0$. Furthermore, for $\varepsilon > 0$ small enough, 
$f(x) = H(\sqrt{a-x})$ on $(a-\varepsilon,a)$ where $H(x)$ is a real 
analytical function near zero, $H(0) = 0$, and 
\[
H'(0) = \frac{1}{\pi a} \sqrt{\frac{2}
{\xII''(\bs\tdl_a)}} \int \frac{t}{(1+\bs\tdl_a t)^2} \nu(dt) .
\]
\end{theorem} 
To prove the theorem, we start with the following lemma which is proven in 
the appendix:
\begin{lemma} 
\label{3rd-deriv} 
Assume that either $\nu$ or $\tilde{\nu}$ is not a Dirac measure.
Let $({\bs\tdl}_a, a)$ with ${a}\neq 0$ satisfy 
\[
\bs F({\bs\tdl}_a,{a})=0,\quad \frac{\partial \bs F}{\partial {\bs\tdl}}
({\bs\tdl}_a,{a}) = 0
\]
where the function $\bs F(\bs\tdl, \bs x)$ is defined by Equation 
\eqref{F-real}. Then 
\[
\frac{\partial^2 \bs F}{\partial {\bs\tdl}^2} ({\bs\tdl}_a,{a}) = 0 
 ~\Rightarrow~ 
\frac{\partial^3 \bs F}{\partial {\bs\tdl}^3} ({\bs\tdl}_a,{a}) \neq 0. 
\]
\end{lemma}

\begin{proof}[Proof of Theorem \ref{sq-root}] 
We follow the argument of \cite{MarPas67}. 
We first assume that $\xII$ reaches a maximum at 
$\bs\tdl_a \in \widetilde{\cal I}$ and prove that $\xII''(\bs\tdl_a) < 0$. 
Observe that $\xII(\bs\tdl)$ satisfies $\bs F(\bs\tdl, \xII(\bs\tdl)) = 0$, 
and that $\partial\bs F/\partial\bs x = 
\int t (\bs x(1 + \bs\delta t))^{-2} \tilde\nu(dt) > 0$. 
By the chain rule for differentiation, 
\begin{align*} 
0 &= \frac{\partial\bs F}{\partial\bs\tdl} + 
\frac{\partial\bs F}{\partial\bs x} \xII'(\bs\tdl), \\
0 &=  
\frac{\partial^2\bs F}{\partial\bs\tdl^2} + 
\left( \frac{\partial^2\bs F}{\partial\bs x^2} + 
2 \frac{\partial^2\bs F}{\partial\bs\tdl\partial\bs x} \right) \xII'(\bs\tdl) 
+ \frac{\partial\bs F}{\partial\bs x} \xII''(\bs\tdl) . 
\end{align*} 
If we assume that $\xII''(\bs\tdl_a) = 0$, then 
$(\partial^2\bs F/\partial\bs\tdl^2)(\bs\tdl_a, a) = 0$ and it is furthermore 
easy to check that 
\[
\bs x^{(3)}(\bs\tdl_a) = - 
\frac{\partial^3\bs F/\partial\bs\tdl^3}
{\partial\bs F/\partial\bs x} (\bs\tdl_a, a) . 
\]
By Lemma \ref{3rd-deriv}, $\bs x^{(3)}(\bs\tdl_a) \neq 0$, but this contradicts
the fact that the first non zero derivative of a function at a local extremum
is of even order. Hence $\xII''(\bs\tdl_a) < 0$. \\
Equation \eqref{eq-x} shows that $\xII$ can be analytically extended to 
a function $\zII$ in a neighborhood of $\bs\tdl_a$ in the complex plane. Since 
$\xII'(\bs\tdl_a) = 0$ and $\xII''(\bs\tdl_a) < 0$, we can write 
$\zII(\tdl) - a = \varphi(\tdl)^2$ in this neighborhood where $\varphi$ is 
an analytical function satisfying $\varphi(\bs\tdl_a) = 0$ and 
$(\varphi'(\bs\tdl_a))^2 = \xII''(\bs\tdl_a) / 2$. We choose $\varphi$ such
that $\varphi'(\bs\tdl_a) = - \imath (-\xII''(\bs\tdl_a) / 2)^{1/2}$. If we 
choose $x > a$ such that $x-a$ is small enough, then 
$\zII(\tdl(x))  - a = \varphi(\tdl(x))^2$, and moreover $\zII(\tdl(x)) = x$.  
Considering the local inverse $\Phi$ of $\varphi$ in a neighborhood of 
$\bs\tdl_a$, we get that $\tdl(x) = \Phi(\sqrt{x-a})$ where the analytic 
function $\Phi$ satisfies 
$\Phi(0) = \bs\tdl_a$ and $\Phi'(0) = 1/ \varphi'(\bs\tdl_a) = 
\imath (-2 / \xII''(\bs\tdl_a))^{1/2}$ (thus the choice of 
$\varphi'(\bs\tdl_a)$ ensures that $\Im \tdl(x) > 0$). Using the equation
$\Im m(x) = - x^{-1} \int \Im ((1 + \tdl(x) t)^{-1}) \nu(dt)$, we get the
result. The case where $\xII$ reaches a minimum at $\bs\tdl_a$ is treated
similarly. 
\end{proof} 

\appendix 
\section{Proof of Lemma \ref{3rd-deriv}} 
\label{prf-der3} 
First recall that
	\begin{align}
		\label{eq:first_der}
		\frac{\partial \bs F}{\partial {\bs\tdl}} ({\bs\tdl},{\bs x}) = {\bs x}^2 {\bs \gamma}({\bs x},{\bs\tdl})\bs{\tilde\gamma}({\bs x},{\bs\delta}) - 1
	\end{align}
	so that ${a}^2 {\bs \gamma}_a\bs{\tilde\gamma}_a=1$, with 
${\bs \gamma}_a={\bs \gamma}({a},{\bs\tdl}_a)$, 
$\bs{\tilde\gamma}_a=\bs{\tilde\gamma}({a},{\bs\delta}_a)$, and 
\[
{\bs\delta}_a=c\int \frac{t}{-{a}(1+{\bs\tdl}_at)}\nu(dt) . 
\] 
Differentiating \eqref{eq:first_der}, the equation 
$(\partial^2 \bs F/\partial {\bs\tdl}^2) ({\bs\tdl}_a,{a}) = 0$ reads
\begin{align}
	\label{eq:sec_der}
	\bs{\tilde\gamma}_a c \int \frac{t^3}{(1+{\bs\tdl}_at)^3}\nu(dt) + {a} {\bs \gamma}_a^2 \int \frac{t^3}{(1+{\bs\delta}_at)^3}\tilde{\nu}(dt) = 0
\end{align}
where we used 
	\begin{align*}
		\frac{\partial}{\partial {\bs\tdl}} \left( c\int \frac{t}{-{\bs x}(1+{\bs\tdl}t)}\nu(dt) \right)({\bs\tdl}_a,{a})={a}{\bs \gamma}_a.
	\end{align*}
	Assume now that $(\partial^3 \bs F/\partial {\bs\tdl}^3) ({\bs\tdl}_a,{a}) = 0$. A second differentiation of \eqref{eq:first_der} leads then to
	\begin{align*}
		0&=2\frac{{\bs \gamma}_a}{a} c \int \frac{t^3}{(1+{\bs\delta}_at)^3}\tilde{\nu}(dt) \int \frac{t^3}{(1+{\bs\tdl}_at)^3}\nu(dt) \nonumber \\
		&+ \bs{\tilde\gamma}_a c \int \frac{t^4}{(1+{\bs\tdl}_at)^4}\nu(dt) + {a}^2{\bs \gamma}_a^3 \int \frac{t^4}{(1+{\bs\delta}_at)^4}\tilde{\nu}(dt).
	\end{align*}
	Using ${a}^2 {\bs \gamma}_a\bs{\tilde\gamma}_a=1$, replace now ${\bs \gamma}_a/{a}$ by $1/({a}^3\bs{\tilde\gamma}_a)$ in the leftmost term and ${a}^2{\bs \gamma}_a^3$ by ${\bs \gamma}_a^2/\bs{\tilde\gamma}_a$ in the rightmost term. Multiplying the result by $\bs{\tilde\gamma}_a$ leads to
	\begin{align}
		\label{eq:3rd_der}
		0&= 2 \frac{c}{ {a}^3 } \int \frac{t^3}{(1+{\bs\delta}_at)^3}\tilde{\nu}(dt) \int \frac{t^3}{(1+{\bs\tdl}_at)^3}\nu(dt) \nonumber \\
		&+ \bs{\tilde\gamma}_a^2 c \int \frac{t^4}{(1+{\bs \tdl}_at)^4}\nu(dt) +  {\bs \gamma}_a^2 \int \frac{t^4}{(1+{\bs \delta}_at)^4}\tilde{\nu}(dt).
	\end{align}
	We now use \eqref{eq:sec_der} and ${a}^2{\bs \gamma}_a\bs{\tilde\gamma}_a=1$ to write the two equations:
	\begin{align*}
		2\frac{c}{ {a}^3 } \int \frac{t^3}{(1+{\bs\tdl}_at)^3}\nu(dt) &= -\frac{2}{ {a}^2 } \frac{ \bs{\gamma}_a^2 }{ \bs{\tilde\gamma}_a } \int \frac{t^3}{(1+\bs{\delta}_at)^3}\tilde{\nu}(dt) \\
		2\frac{c}{ {a}^3 } \int \frac{t^3}{(1+{\bs\delta}_at)^3}\tilde{\nu}(dt) &= -\frac{2c^2}{ {a}^2 } \frac{ \bs{\tilde\gamma}_a^2 }{ {\bs \gamma}_a } \int \frac{t^3}{(1+{\bs\tdl}_at)^3}\nu(dt).
	\end{align*}
	Replacing the corresponding terms in the leftmost term of \eqref{eq:3rd_der} leads to the two equations
	\begin{align*}
		\frac{2}{ {a}^2 } \frac{ {\bs \gamma}_a^2 }{ \bs{\tilde\gamma}_a } \left( \int \frac{t^3}{(1+{\bs\delta}_at)^3}\tilde{\nu}(dt) \right)^2 - \bs{\tilde\gamma}_a^2 c\int \frac{t^4}{(1+{\bs \tdl}_at)^4}\nu(dt) - {\bs \gamma}_a^2 \int \frac{t^4}{(1+{\bs \delta}_at)^4}\tilde{\nu}(dt)&=0 \\
		\frac{2}{ {a}^2 } \frac{ \bs{\tilde\gamma}_a^2 }{ {\bs\gamma}_a } \left( c \int \frac{t^3}{(1+{\bs\tdl}_at)^3}\nu(dt) \right)^2 - \bs{\tilde\gamma}_a^2 c \int \frac{t^4}{(1+{\bs \tdl}_at)^4}\nu(dt) - {\bs \gamma}_a^2 \int \frac{t^4}{(1+{\bs \delta}_at)^4}\tilde{\nu}(dt)&=0.
	\end{align*}
	Multiplying each equation by ${\bs \gamma}_a\bs{\tilde\gamma}_a$ and averaging then gives:
\begin{align}
	\label{eq:3rd_der_2}
	0&=\frac1{ {a}^2 } {\bs \gamma}_a^3 \left( \int \frac{t^3}{(1+{\bs\delta}_at)^3}\tilde{\nu}(dt) \right)^2 + \frac1{ {a}^2 } \bs{\tilde\gamma}_a^3 \left( c \int \frac{t^3}{(1+{\bs\tdl}_at)^3}\nu(dt) \right)^2 \nonumber \\
	&-  \bs{\tilde\gamma}_a^3{\bs \gamma}_a c\int \frac{t^4}{(1+{\bs \tdl}_at)^4}\nu(dt) - {\bs \gamma}_a^3\bs{\tilde\gamma}_a \int \frac{t^4}{(1+{\bs \delta}_at)^4}\tilde{\nu}(dt).
\end{align}
Remark now, by expanding the definition of ${\tilde\gamma}_a$ that
\begin{align*}
	&\frac1{ {a}^2 } {\bs \gamma}_a^3 \left( \int \frac{t^3}{(1+{\bs\delta}_at)^3}\tilde{\nu}(dt) \right)^2 - {\bs \gamma}_a^3\bs{\tilde\gamma}_a \int \frac{t^4}{(1+{\bs \delta}_at)^4}\tilde{\nu}(dt) \nonumber \\
	&= \frac{{\bs \gamma}_a^3}{ {a}^2}\left[ \left( \int \frac{t^3}{(1+{\bs\delta}_at)^3}\tilde{\nu}(dt) \right)^2 -  \int \frac{t^2}{(1+{\bs\delta}_at)^2}\tilde{\nu}(dt) \int \frac{t^4}{(1+{\bs \delta}_at)^4}\tilde{\nu}(dt) \right] \\
	&\leq 0
\end{align*}
with the inequality arising from Cauchy--Schwarz. The case of equality holds only if $\tilde\nu$ is a Dirac measure. Similarly, 
\begin{align*}
	&\frac1{ {a}^2 } \bs{\tilde\gamma}_a^3 \left( c \int \frac{t^3}{(1+{\bs\tdl}_at)^3}\nu(dt) \right)^2 - \bs{\tilde\gamma}_a^3{\bs \gamma}_a c\int \frac{t^4}{(1+{\bs\tdl}_at)^4}\nu(dt) \nonumber \\
	&= \frac{\bs{\tilde\gamma}_a^3}{ {a}^2}\left[ \left( c\int \frac{t^3}{(1+{\bs\tdl}_at)^3}\nu(dt) \right)^2 -  \left(c\int \frac{t^2}{(1+{\bs\tdl}_at)^2}\nu(dt)\right) \left( c\int \frac{t^4}{(1+{\bs \tdl}_at)^4}\nu(dt)\right) \right]\\
	&\leq 0
\end{align*}
with equality only if $\nu$ is a Dirac measure. Therefore, to ensure \eqref{eq:3rd_der_2}, both $\nu$ and $\tilde{\nu}$ must be Dirac measures, which goes against the hypothesis.


\def\cprime{$'$} \def\cdprime{$''$} \def\cprime{$'$} \def\cprime{$'$}
  \def\cprime{$'$} \def\cprime{$'$}

\end{document}